\documentclass[12pt]{article}
\usepackage{amssymb}
\usepackage{mathrsfs}
\usepackage[hypertex]{hyperref}
\usepackage{amsfonts}
\usepackage{graphicx}
\usepackage{mathptmx}
\usepackage{latexsym,amsmath,amssymb,amsfonts,amsthm}

\newtheorem{theorem}{Theorem}[section]
\newtheorem{lemma}[theorem]{Lemma}

\newtheorem{corollary}[theorem]{Corollary}

\newtheorem{remark}[theorem]{Remark}

\begin{document}
\setcounter{page}{1}
\title{Translating solutions of the nonparametric mean curvature flow with nonzero Neumann boundary data in product manifold $M^{n}\times\mathbb{R}$}
\author{Ya Gao,~~Yi-Juan Gong,~~Jing Mao$^{\ast}$}
\date{}
\protect\footnotetext{\!\!\!\!\!\!\!\!\!\!\!\!{$^{\ast}$Corresponding
author}\\
{MSC 2020:} 53C42, 53B20, 35B50, 35K93.
\\
{Key Words:} Translating solutions, singularity, nonparametric mean
curvature flow, convexity, Ricci curvature.}
\maketitle ~~~\\[-15mm]
\begin{center}{\footnotesize
Faculty of Mathematics and Statistics, \\
Key Laboratory of Applied Mathematics of Hubei Province, \\
Hubei University, Wuhan 430062, China\\
Emails: jiner120@163.com, jiner120@tom.com }
\end{center}

\begin{abstract}
In this paper, we can prove the existence of translating solutions
to the nonparametric mean curvature flow with \emph{nonzero} Neumann
boundary data in a prescribed product manifold
$M^{n}\times\mathbb{R}$, where $M^{n}$ is an $n$-dimensional
($n\geq2$) complete Riemannian manifold with nonnegative Ricci
curvature, and $\mathbb{R}$ is the Euclidean $1$-space.
\end{abstract}

\markright{\sl\hfill  Y. Gao, Y.-J. Gong, J. Mao\hfill}

\section{Introduction}
\renewcommand{\thesection}{\arabic{section}}
\renewcommand{\theequation}{\thesection.\arabic{equation}}
\setcounter{equation}{0} \setcounter{maintheorem}{0}

The mean curvature flow (MCF for short) is one of the most important
\emph{extrinsic} curvature flows and has many nice applications. For
instance, by using the curve shortening flow (i.e., the
lower-dimensional case of MCF), Topping \cite{pt} successfully gave
an isoperimetric inequality on simply connected surfaces with
Gaussian curvature satisfying some integral precondition. This
result
 extends those isoperimetric inequalities (introduced in
detail in, e.g., \cite{b1,o1}) obtained separately by Alexandrov,
Fiala-Huber, Bol, and Bernstein-Schmidt. Applying the long-time
existence and convergence conclusions of graphic MCF of any
codimension in prescribed product manifolds (see \cite{mtw1}), Wang
\cite{mtw2} showed that for a bounded $C^2$ convex domain
$\mathbb{D}$ (with diameter $\delta$ and boundary
$\partial\mathbb{D}$) in the Euclidean $n$-space $\mathbb{R}^n$ and
$\phi:\partial\mathbb{D}\rightarrow\mathbb{R}^m$ a continuous map,
there exists a map $\psi:\mathbb{D}\rightarrow\mathbb{R}^m$, with
$\psi|_{\partial\mathbb{D}}=\phi$ and with the graph of $\psi$ a
minimal submanifold in $\mathbb{R}^{n+m}$, provided
$\psi|_{\partial\mathbb{D}}$ is a smooth map and
$8n\delta\sup_{\mathbb{D}}|D^{2}\psi|+\sqrt{2}\sup_{\partial\mathbb{D}}|D\psi|<1$.
This conclusion provides classical solutions to the \emph{Dirichlet
problem for minimal surface systems in arbitrary codimensions} for a
class of boundary maps. Specially, when $m=1$, the existence of
$\psi$ was obtained by Jenkins and Serrin \cite{js} already.
Inspired by Wang's work mentioned above, by applying the spacelike
MCF in the Minkowski space $\mathbb{R}^{n+m,n}$, Mao \cite{jm1} can
successfully get the existence of $\psi$ for maximal spacelike
submanifolds (with index $n$) in $\mathbb{R}^{n+m,n}$.

In the early study of the theory of MCF, a classical result from
Huisken \cite{gh1} says that a given compact strictly convex
hypersurface $M^n$ in $\mathbb{R}^{n+1}$ evolving along the MCF
would contract to a single point at finite time. More precisely, let
$X(\cdot,t)=X_{t}$ be a one-parameter family of immersions
$X_{t}:M^{n}\rightarrow\mathbb{R}^{n+1}$ whose images
$M_{t}^{n}=X_{t}(M^{n})$ satisfy
\begin{eqnarray} \label{MCF}
 \frac{\partial}{\partial t}X(x,t)=\vec{H},\qquad \qquad &\mathrm{on} ~ M^{n}\times[0,T)
 \end{eqnarray}
for some $T>0$, with the initial condition $X(x,0)=X_{0}(x)$ on
$M^n$, where $\vec{H}$ is the mean curvature vector of the evolving
hypersurface $M_{t}^{n}$, by using the method of $L^{p}$ estimates,
Huisken \cite{gh1} proved that if $M^{n}$ is a compact strictly
convex hypersurface in $\mathbb{R}^{n+1}$, the MCF equation
(\ref{MCF}), with the initial condition, has a unique smooth
solution on the finite time interval $[0,T_{\max})$ with
$T_{\max}<\infty$, and the evolving hypersurfaces $M^{n}_{t}$
contract to a single point as $t\rightarrow T_{\max}$. By imposing a
pinching condition on the second fundamental form of the initial
hypersurface, Huisken \cite{gh2} has extended the above conclusion
to a more general setting that the ambient space $\mathbb{R}^{n+1}$
was replaced by a class of smooth complete Riemannian manifolds
$N^{n+1}$ having some uniform bounds for curvatures and injectivity
radius (of course, $N^{n+1}$ covers $\mathbb{R}^{n+1}$ as a special
case). From these two facts, one might know that generally the MCF
would occur singularity at finite time. A natural question is:

\vspace{5mm}

\textbf{Problem 1}. \emph{When does the MCF exist for all the time?}

 $\\$That is to say, under specified settings, there is no singularity
formed during the evolution of MCF.

If there exists a constant vector $V$ such that
\begin{eqnarray*}
\vec{H}=V^{\perp},
\end{eqnarray*}
then the evolving submanifold $X_{t}:M^{n}\rightarrow
\mathbb{R}^{n+m}$ is called a \emph{translating soliton} of the MCF
equation (\ref{MCF}). Here $(\cdot)^{\perp}$ denotes the normal
projection of a prescribed vector to the normal bundle of
$M_{t}^{n}$ in $\mathbb{R}^{m+n}$. It is easy to see that the
translating soliton gives an eternal solution $X_{t}=X_{0}+tV$ to
(\ref{MCF}), which is called the translating solution. Translating
solitons play an important role in the study of type-II
singularities of the MCF. For instance, Angenent and Vel\'{a}zquez
\cite{av1,av2} gave some examples of convergence which implies that
type-II singularities of the MCF there are modeled by translating
surfaces. Clearly, the existence of translation solutions to the
equation (\ref{MCF}) can give a positive answer to \textbf{Problem
1}.

Huisken \cite{gh3} considered the evolution of graphic hypersurfaces
 over a bounded domain (with smooth boundary) in $\mathbb{R}^{n}$
 under the MCF with a vanishing Neumann boundary condition (NBC for
 short), and proved that the flow exists for all the time and
 evolving graphic hypersurfaces in $\mathbb{R}^{n+1}$ converge to the graph of a
 constant function as $t\rightarrow\infty$. The vanishing NBC here
 has strong geometric meaning, that is, \emph{the evolving graphic hypersurface is perpendicular with the parabolic boundary during the
 evolution (or the contact angle between the evolving graphic hypersurface and parabolic boundary is $\pi/2$)}. Is the vanishing NBC necessary? What about the
 non-vanishing case? There are many literatures working on this
 direction and we would like to mention some of them. When the dimension $n$ satisfies $n=1$ or $n=2$, Altschuler
 and Wu \cite{aw,aw1} gave a positive answer to these questions. In
 fact, they proved:

 \begin{itemize}

\item when $n=1$, a graphic curve defined over an open bounded interval
evolves along the flow given by a class of quasilinear parabolic
equations (of course, including the MCF as a special case), with
arbitrary contact angle (i.e., with nonzero NBC), would exist for
all the time, and the evolving curves converge  as
$t\rightarrow\infty$ to a solution moving by translation with speed
uniquely determined by the boundary data.

\item when $n=2$, a graphic surface defined over a compact strictly
convex domain (with smooth boundary) in $\mathbb{R}^2$ evolves along
the MCF, with arbitrary contact angle (i.e., with nonzero NBC),
would exist for all the time, and the evolving surfaces converge as
$t\rightarrow\infty$ to a surface (unique up to translation) which
moves at a constant speed (uniquely determined by the boundary
data).

 \end{itemize}
For the higher dimensional case, Guan \cite{gb} have given a partial
answer. In fact, he can get the long-time existence of the evolution
of graphic hypersurfaces, defined over a bounded domain (with smooth
boundary) in $\mathbb{R}^{n}$,
 under a nonparametric mean curvature type flow (i.e., the MCF with a forcing term given by an \emph{admissible function} defined therein) with nonzero NBC. However, the asymptotic behavior of
 the flow cannot be obtained in his setting.  Zhou \cite{hyz} extended Altschuler-Wu's conclusion \cite{aw1} to the situation that graphic surfaces were defined over a compact strictly
convex domain (with smooth boundary) in $2$-dimensional Riemannian
surfaces $M^2$ with nonnegative Ricci curvature, and extended Guan's
conclusion \cite{gb} to the situation that graphic hypersurfaces
were defined over a bounded domain (with smooth boundary) in
$n$-dimensional ($n\geq2$) Riemannian manifolds $M^{n}$. However,
similar to Guan's work \cite{gb}, Zhou \cite{hyz} \emph{also} cannot
give the asymptotic behavior of the MCF with a forcing term (given
by an \emph{admissible function}) and with nonzero NBC in product
manifolds $M^{n}\times\mathbb{R}$. Recently, Ma, Wang and Wei
\cite{mww} improved Huisken's work \cite{gh3} to a more general
setting that the vanishing NBC therein can be replaced by a nonzero
NBC of specialized type.

Our purpose here is trying to extend the main conclusion in
\cite{mww} to a more general case -- the ambient space
$\mathbb{R}^{n+1}$ will be replaced by product manifolds of type
$M^{n}\times\mathbb{R}$, where $M^{n}$ is a complete Riemannian
manifold of nonnegative Ricci curvature.

Throughout this paper, let $(M^{n},\sigma)$ be a complete
$n$-manifold ($n\geq2$) with the Riemannian metric $\sigma$, and let
$\Omega\subset M^{n}$ be a compact strictly convex domain with
smooth boundary $\partial\Omega$. Denote by
$\left(U_{A};w^{1}_{A},w^{2}_{A},\cdots,w^{n}_{A}\right)$ the local
coordinate coverings of $M$, and $\frac{\partial}{\partial
w^{i}_{A}}$, $i=1,2,\cdots,n$, the corresponding coordinate vector
fields, where $A\in I\subseteq N$ with $N$ the set of all positive
integers. For simplicity, we just write
$\{w^{1}_{A},w^{2}_{A},\cdots,w^{n}_{A}\}$ as
$\{w^{1},w^{2},\cdots,w^{n}\}$ to represent the local coordinates on
$M$, and write $\frac{\partial}{\partial w^{i}_{A}}$ as
$\frac{\partial}{\partial w^{i}}$ or $\partial_{i}$. In this
setting, the metric $\sigma$ should be
$\sigma=\sum_{i,j=1}^{n}\sigma_{ij}dw^{i}\otimes dw^{j}$ with
$\sigma_{ij}=\sigma(\partial_{i},\partial_{j})$. Denote by $D$,
$D^{\partial\Omega}$ the covariant derivatives on $\Omega$ and
$\partial\Omega$ respectively. Now, we would like to consider, along
the MCF (\ref{MCF}) with nonzero NBC, the evolution of graphic
hypersurfaces, defined over $\Omega$, in product manifold
$M^{n}\times\mathbb{R}$ with the product metric
$\overline{g}=\sigma_{ij}dw^{i}\otimes dw^{j}+ds\otimes ds$. More
precisely, given a smooth\footnote{In fact, it is not necessary to
impose smoothness assumption on the initial hypersurface
$\mathcal{G}$. The $C^{2,\alpha}$-regularity for $\mathcal{G}$ is
enough to get all the estimates in the sequel. However, in order to
avoid the boring regularity arguments, which is not necessary, here
we assume $\mathcal{G}$ is smooth.} graphic hypersurface
$\mathcal{G}\subset M^{n}\times\mathbb{R}$ defined over $\Omega$,
then there exists a smooth function $u_{0}\in
C^{\infty}(\overline{\Omega})$ such that $\mathcal{G}$ can be
represented by $\mathcal{G}:=\{(x,u_{0}(x))|x\in\Omega\}$. It is not
hard to know that the metric of  $\mathcal{G}$ is given by
$g=i^{\ast}\overline{g}$, where $i^{\ast}$ is the pullback mapping
of the immersion $i:\mathcal{G}\hookrightarrow
M^{n}\times\mathbb{R}$, tangent vectors are given by
\begin{eqnarray*}
\vec{e_{i}}=\partial_{i}+D_{i}u\partial_{s}, \qquad  i=1,2,\cdots,n,
\end{eqnarray*}
 and the corresponding upward unit normal vector is given by
\begin{eqnarray*}
\vec{\gamma}=-\frac{\sum\limits_{i=1}^{n}D^{i}u\partial_{i}-\partial_s}{\sqrt{1+|Du|^2}},
\end{eqnarray*}
where $D^{j}u=\sum_{i=1}^{n}\sigma^{ij}D_{i}u$. Denote by $\nabla$
the covariant derivative operator on $M^{n}\times\mathbb{R}$, and
then the second fundamental form $h_{ij}d\omega^{i}\otimes
d\omega^{j}$ of $\mathcal{G}$ is given by
\begin{eqnarray*}
h_{ij}=-\langle\nabla_{\vec{e}_i}\vec{e}_j,\vec{\gamma}\rangle_{\overline{g}}=-\frac{D_{i}D_{j}u}{\sqrt{1+|Du|^2}}.
\end{eqnarray*}
Moreover, the scalar mean curvature of $\mathcal{G}$ is
\begin{eqnarray} \label{smc}
\qquad
H=\sum_{i=1}^{n}h^i_i=-\frac{\sum\limits_{i,k=1}^{n}g^{ik}D_{i}D_{k}u}{\sqrt{1+|Du|^2}}=-\frac{\sum\limits_{i,k=1}^{n}\left(\sigma^{ik}-\frac{D^{i}uD^{k}u}{1+|Du|^{2}}\right)D_{i}D_{k}u}{\sqrt{1+|Du|^2}}.
\end{eqnarray}
Hence, in our situation here, the evolution of $\mathcal{G}$ under
the MCF with nonzero NBC in $M^{n}\times\mathbb{R}$ with the metric
$\overline{g}$ can be reduced to solvability of the following
initial-boundary value problem (IBVP for short)
\begin{eqnarray*}
(\sharp)\qquad \left\{
\begin{array}{lll}
{\frac{\partial u}{\partial t}=\sum\limits_{i,j=1}^{n}\left(\sigma^{ij}-\frac{D^{i}uD^{j}u}{1+|Du|^{2}}\right)D_{i}D_{j}u} \qquad \qquad &\mathrm{in} ~ \Omega\times[0,T),\\
D_{\vec{\nu}}u=\phi(x) \qquad \qquad &\mathrm{on} ~
\partial \Omega\times[0,T),\\
u(\cdot,0)=u_{0}(\cdot)\qquad \qquad &\mathrm{on} ~\Omega_{0},
\end{array}
\right.
\end{eqnarray*}
where $\vec{\nu}$ is the inward unit normal vector of
$\partial\Omega$, $\Omega_{t}=\Omega\times\{t\}$ is a slice in
$\Omega\times[0,T)$, $u_{0}(x)\in C^{\infty}(\overline{\Omega})$ and
$\phi(x)\in C^{\infty}(\overline{\Omega})$ are smooth function
satisfying
\begin{eqnarray} \label{cc}
u_{0,\vec{\nu}}=\phi(x)\qquad &\mathrm{on} ~\partial\Omega.
 \end{eqnarray}
Here (\ref{cc}) is called \emph{compatibility condition} of the
system ($\sharp$), and a comma ``," in the subscript means doing
covariant derivative w.r.t. a prescribed tensor. This convention
will also be used in the sequel. For the IBVP ($\sharp$), we can
prove:

\begin{theorem} \label{theorem1-1}
If the Ricci curvature of $M^{n}$ is nonnegative, then for the IBVP
($\sharp$), we have

(1) the IBVP ($\sharp$) has a smooth solution $u(x,t)$ on
$\overline{\Omega}\times[0,\infty)$;

(2) the smooth solution $u(x,t)$ converges as $t\rightarrow\infty$
to $\lambda t+w(x)$, i.e.,
\begin{eqnarray*}
\lim\limits_{t\rightarrow\infty}\|u(x,t)-(\lambda
t+w(x))\|_{C^{0}(\overline{\Omega})}=0,
\end{eqnarray*}
where $\lambda\in\mathbb{R}$ and $w\in
C^{2,\alpha}(\overline{\Omega})$ (unique up to a constant) solving
the following boundary value problem (BVP for short)
\begin{eqnarray*}
(\ddag)\qquad \left\{
\begin{array}{ll}
\sum\limits_{i,j=1}^{n}\left(\sigma^{ij}-\frac{D^{i}uD^{j}u}{1+|Du|^{2}}\right)D_{i}D_{j}u=\lambda \qquad \qquad &\mathrm{in} ~ \Omega,\\
D_{\vec{\nu}}u=\phi(x) \qquad \qquad &\mathrm{on} ~
\partial \Omega.
\end{array}
\right.
\end{eqnarray*}
Here $0<\alpha<1$ and $\lambda$ is called the additive eigenvalue of
the BVP ($\ddag$).
\end{theorem}

\begin{remark}
\rm{ (I) By (\ref{smc}), it is easy to know that
\begin{eqnarray*}
\sum\limits_{i,j=1}^{n}\left(\sigma^{ij}-\frac{D^{i}uD^{j}u}{1+|Du|^{2}}\right)D_{i}D_{j}u=H\cdot\sqrt{1+|Du|^{2}}=\mathrm{div}\left(\frac{Du}{\sqrt{1+|Du|^{2}}}\right)\cdot\sqrt{1+|Du|^{2}},
\end{eqnarray*}
which, substituting into the first equation of ($\ddag$), implies
\begin{eqnarray*}
\mathrm{div}\left(\frac{Dw}{\sqrt{1+|Dw|^{2}}}\right)=\frac{\lambda}{\sqrt{1+|Dw|^{2}}},
\end{eqnarray*}
where $u=w(x)$ is the solution to the BVP ($\ddag$). Integrating the
above equality and using the divergence theorem, one can get
\begin{eqnarray*}
\lambda=-\frac{\int_{\partial\Omega}\frac{\phi(x)}{\sqrt{1+|Dw|^{2}}}}{\int_{\Omega}(1+|Dw|^{2})^{-\frac{1}{2}}}.
\end{eqnarray*}
Clearly, if $\phi(x)\equiv0$, then $\lambda=0$. Moreover, in this
setting, for the IBVP ($\sharp$), as $t\rightarrow\infty$, its
smooth solution $u(x,t)$ would converge to a constant function
defined over
$\Omega\subset M^{n}$. \\
 (II) We would like to mention one
thing, that is, if $M^{n}=\mathbb{R}^n$ and $\phi(x)\equiv0$, then
Theorem \ref{theorem1-1} here degenerates into Huisken's main
conclusion in \cite{gh3}; if $M^{n}=\mathbb{R}^n$, our main conclusion here becomes exactly \cite[Theorems 1.1 and 1.2]{mww}.\\
(III) Recent years, the study of submanifolds of constant curvature
in product manifolds attracts many geometers' attention. For
instance, Hopf in 1955 discovered that the complexification of the
traceless part of the second fundamental form of an immersed surface
$\Sigma^{2}$, with constant mean curvature (CMC for short) $H$, in
$\mathbb{R}^3$ is a holomorphic quadratic differential $Q$ on
$\Sigma^{2}$, and then he used this observation to get his
well-known conclusion that any immersed CMC sphere
$\mathbb{S}^{2}\hookrightarrow\mathbb{R}^3$ is a standard distance
sphere with radius $1/H$. By introducing a generalized quadratic
differential $\widetilde{Q}$ for immersed surfaces $\Sigma^{2}$ in
product spaces $\mathbb{S}^{2}\times\mathbb{R}$ and
$\mathbb{H}^{2}\times\mathbb{R}$, with $\mathbb{S}^{2}$,
$\mathbb{H}^{2}$ the $2$-dimensional sphere and hyperbolic surface
respectively, Abresch and Rosenberg \cite{ar} can extend Hopf's
result to CMC spheres in these target spaces. Meeks and Rosenberg
\cite{mr} successfully classified stable properly embedded
orientable minimal surfaces in the product space
$M\times\mathbb{R}$, where $M$ is a closed orientable Riemannian
surface. In fact, they proved that such a surface must be a product
of a stable embedded geodesic on $M$ with $\mathbb{R}$, a minimal
graph over a region of $M$ bounded by stable geodesics,
$M\times\{t\}$ for some $t\in\mathbb{R}$, or is in a moduli space of
periodic multigraphs parameterized by $P\times\mathbb{R}^{+}$, where
$P$ is the set of primitive (non-multiple) homology classes in
$H_{1}(M)$. Mazet, Rodr\'{\i}guez and Rosenberg \cite{lmh} analyzed
properties of periodic minimal or constant mean curvature surfaces
in the product manifold $\mathbb{H}^{2}\times\mathbb{R}$, and they
also construct examples of periodic minimal surfaces in
$\mathbb{H}^{2}\times\mathbb{R}$. In \cite{hfj}, Rosenberg, Schulze
and Spruck showed that a properly immersed minimal hypersurface in
$M\times\mathbb{R}^{+}$ equals some slice $M\times\{c\}$ when $M$ is
a complete, recurrent $n$-dimensional Riemannian manifold with
bounded curvature. Of course, for more information, readers can
check references therein of these papers. Hence, it is interesting
and important to consider submanifolds of constant curvature in the
product manifold of type $M^{n}\times\mathbb{R}$. Based on this
reason, in our setting here, it should be interesting and important
to consider the following CMC equation with nonzero NBC
\begin{eqnarray*}
(\natural)\qquad \left\{
\begin{array}{ll}
H=\mathrm{div}\left(\frac{Du}{\sqrt{1+|Du|^{2}}}\right)=\lambda \qquad \qquad &\mathrm{in} ~ \Omega,\\
D_{\vec{\nu}}u=\phi(x) \qquad \qquad &\mathrm{on} ~
\partial \Omega.
\end{array}
\right.
\end{eqnarray*}
Of course, all the symbols in the above system have the same meaning
as those in ($\ddag$). The existence and uniqueness of solution to
the BVP ($\natural$) have been obtained recently -- see \cite{gms}
for details. \\
(IV) The evolution of space-like surfaces in the Lorentz
$3$-manifold $M^{2}\times\mathbb{R}$ under the MCF with arbitrary
contact angle (of course, in this situation, the NBC is nonzero) has
been investigated in \cite{chmx}, and the long-time existence and
the existence of translating solutions to the flow have been
obtained. \\
(V) As we know, if the warping function was chosen to be a constant
function, then warped product manifolds would degenerate into
product manifolds. Hence, one might ask ``\emph{whether one could
expect to get a similar conclusion to Theorem \ref{theorem1-1} in
warped product manifolds or not?"}. By constructing an interesting
graphic hypersurface example in a prescribed warped product (see
\cite[Appendix A]{hyz}), Zhou gave a negative answer to this
question. Speaking in other words, he showed that the MCF with
nonzero NBC in
warped product manifolds would form singularities within finite time. \\
(VI). In fact, Huisken \cite{gh3} considered the following IBVP
\begin{eqnarray*}
\left\{
\begin{array}{lll}
{\frac{\partial u}{\partial t}=\sum\limits_{i,j=1}^{n}\left(\delta^{ij}-\frac{D^{i}uD^{j}u}{1+|Du|^{2}}\right)D_{i}D_{j}u} \qquad \qquad &\mathrm{in} ~ \Omega\times[0,T),\\
D_{\vec{\nu}}u=0 \qquad \qquad &\mathrm{on} ~
\partial \Omega\times[0,T),\\
u(\cdot,0)=u_{0}(\cdot)\qquad \qquad &\mathrm{on} ~\Omega_{0},
\end{array}
\right.
\end{eqnarray*}
which, as mentioned before, describes the evolution of graphic
hypersurfaces
 over $\Omega\subset\mathbb{R}^{n}$
 under the MCF with a zero NBC, and obtained the long-time existence,
 i.e., $T=\infty$. The vanishing NBC here means that
  \begin{eqnarray*}
  \langle\vec{\gamma},\vec{\nu}\rangle_{\overline{g}}=\frac{D_{\vec{\nu}}u}{\sqrt{1+|Du|^2}}=0,
  \end{eqnarray*}
  which is to say $\vec{\gamma}\perp\vec{\nu}$, i.e., the contact angle between $\vec{\gamma}$ and $\vec{\nu}$ is
  $\pi/2$. If the contact angle is arbitrary, then the corresponding
  NBC should have the form
  $D_{\vec{\nu}}u\big{|}_{\partial\Omega}=\varphi(x)\cdot\sqrt{1+|Du|^2}$  for some $\varphi(x)\in
  C^{\infty}(\overline{\Omega})$, $|\varphi(x)|\leq1$ on
  $\partial\Omega$, and
  $\varphi(x)=u_{0,\vec{\nu}}$ on
  $\partial\Omega$. Based on this reason, we can say that although
  the IBVP ($\sharp$) has nonzero NBC, the geometric meaning of the
  NBC in ($\sharp$) is not sufficient. \emph{Can we deal with the IBVP
  ($\sharp$) if the RHS of the nonzero NBC therein contains $Du$
  also?} Inspired by a recent work \cite{wwx}, Gao and Mao
  \cite{gm2} considered a generalization of the IBVP
  ($\sharp$) where the NBC can be replaced by
  \begin{eqnarray*}
D_{\vec{\nu}}u=\phi(x)\cdot\left(\sqrt{1+|Du|^2}\right)^{\frac{1-q}{2}}
  \end{eqnarray*}
for any $q>0$, and similar conclusion to Theorem \ref{theorem1-1}
could be derived.
 }
\end{remark}

This paper is organized as follows. In Section 2, the
time-derivative estimate, the gradient estimate, and the estimate
for higher-order derivatives of $u$ will be shown in detail, which
naturally lead to the long-time existence of the IBVP ($\sharp$). In
section 3, by using an approximating approach, the solvability of
the BVP ($\ddag$) can be given first, which will be used later to
get the asymptotic behavior of solutions $u$ to the IBVP ($\sharp$).

\section{The long-time existence}
\renewcommand{\thesection}{\arabic{section}}
\renewcommand{\theequation}{\thesection.\arabic{equation}}
\setcounter{equation}{0} \setcounter{maintheorem}{0}

For convenience, we use several notations as follows:
$$v=\sqrt {1+|Du|^{2}},$$
$$g_{ij}=\sigma_{ij}+D_{i}uD_{j}u,$$
$$g^{ij}=\sigma^{ij}-\frac{D^{i}uD^{j}u}{1+|Du|^{2}},$$
$$u_{t}=\frac{\partial u}{\partial t}.$$
For vectors, $V$, $W$ or matrices $A$, $B$, we shall use the
shorthand as follows:
 \begin{eqnarray*}
 \langle V,W \rangle
_{g}=\sum\limits_{i,j=1}^{n}g^{ij}V_{i}W_{j},\quad \langle
V,W\rangle
_{\sigma}=\sum\limits_{i,j=1}^{n}\sigma^{ij}V_{i}W_{j},\quad \langle
A,B\rangle
_{g,\sigma}=\sum\limits_{i,j,k,l=1}^{n}g^{ij}\sigma^{kl}A_{ik}B_{jl}.
 \end{eqnarray*}

First, by applying a similar method to that in the proof of
\cite[Lemma 2.2]{aw1}, we would like to show the time-derivative
estimate for $u$.

\begin{lemma} \label{lemma2.1}
For the IBVP ($\sharp$), we have
 \begin{eqnarray*}
\sup\limits_{\overline{\Omega}\times[0,T]}|u_{t}|^{2}=\sup\limits_{\Omega_{0}}|u_{t}|^{2}.
 \end{eqnarray*}
That is to say, there exists some positive constant
$c_{0}=c_{0}(u_{0})\in\mathbb{R}^{+}$ such that for any
$(x,t)\in\overline{\Omega}\times[0,T]$, we have
\begin{eqnarray*}
|u_{t}|^{2}(x,t)\leq c_{0}.
\end{eqnarray*}
\end{lemma}

\begin{proof}
We first show that the maximum of $u_{t}$ must occur on
$\left(\partial\Omega\times[0,T]\right)\cup \Omega_{0}$. By a direct
computation, we have
\begin{eqnarray*}
\frac{\partial}{\partial t} |u_{t}|^{2}&=&2u_{t}\frac{\partial u_{t}}{\partial t}\\
&=&
\sum\limits_{i,j=1}^{n}2u_{t}\left(\frac{\partial g^{ij}}{\partial t}D_{i}D_{j}u+g^{ij}D_{i}D_{j}u_{t}\right)\\
&=&
\sum\limits_{i,j,k=1}^{n}2u_{t} \frac{\partial g^{ij}}{\partial D^{k}u}\frac{\partial D^{k}u}{\partial t}D_{i}D_{j}u+\sum\limits_{i,j=1}^{n}g^{ij}(D_{i}D_{j}|u_{t}|^{2}-2D_{i}u_{t}D_{j}u_{t})\\
&=&
\sum\limits_{i,j,k=1}^{n}2u_{t} \frac{\partial g^{ij}}{\partial D^{k}u}\frac{\partial D^{k}u}{\partial t}D_{i}D_{j}u+\sum\limits_{i,j=1}^{n}g^{ij}D_{i}D_{j}|u_{t}|^{2}-2\langle Du_{t},Du_{t}\rangle_{\sigma}\\
&=&
\sum\limits_{i,j,k,l=1}^{n}2u_{t}\frac{\partial g^{ij}}{\partial D^{k}u}\frac{\partial(\sigma^{kl}D_{l}u)}{\partial t}D_{i}D_{j}u+\sum\limits_{i,j=1}^{n}g^{ij}D_{i}D_{j}|u_{t}|^{2}-2\langle Du_{t},Du_{t}\rangle_{\sigma}\\
&=& \sum\limits_{i,j,k,l=1}^{n}\frac{\partial g^{ij}}{\partial
D^{k}u}\sigma^{kl}D_{i}D_{j}uD_{l}|u_{t}|^{2}+\sum\limits_{i,j=1}^{n}g^{ij}D_{i}D_{j}|u_{t}|^{2}-2\langle
Du_{t},Du_{t}\rangle_{\sigma},
\end{eqnarray*}
which implies
\begin{eqnarray*}
 \sup\limits_{\overline{\Omega}\times[0,T]}|u_{t}|^{2}=\sup\limits_{\left(\partial\Omega\times[0,T]\right)\cup
\Omega_{0}}|u_{t}|^{2}
\end{eqnarray*}
 by directly applying the weak maximum
principle.

Next, we expel the possibility that the maximum occurs at
$\partial\Omega\times[0,T]$. Assume that
 \begin{eqnarray*}
\max\limits_{\Omega\times\{t\}}|u_{t}|^{2}=|u_{t}|^{2}(\xi,\tau)>0
 \end{eqnarray*}
for some $(\xi,\tau)\in\partial\Omega\times[0,T]$. By the Hopf
Lemma, it follows that $\frac{\partial|u_{t}|^{2}}{\partial
{\vec{\nu}}}{\big{|}}_{(\xi,\tau)}<0$. But by the boundary condition
of the IBVP ($\sharp$), one has $\frac{\partial|u_{t}|^2}{\partial
{\vec{\nu}}}{\big{|}}_{(\xi,\tau)}=\frac{\partial}{\partial
t}\left(D_{{\vec{\nu}}}u{\big{|}}_{(\xi,\tau)}\right)=\frac{\partial}{\partial
t}(\phi(x))=0$. It is a contradiction. Therefore, the maximum cannot
be achieved at $\partial\Omega\times[0,T]$. The conclusion of Lemma
\ref{lemma2.1} follows.
\end{proof}

We know that if $\Omega$ is a strictly convex domain with smooth
boundary $\partial\Omega$, then there exists a smooth function
$\beta$ on $\Omega$ such that $\beta|_{\Omega}<0$,
$\beta|_{\partial\Omega}=0$, ${\sup_\Omega}|D\beta|\leq1$,
 \begin{eqnarray*}
 \left(\beta_{ij}\right)_{n\times n}\geq
k_{0}\left(\delta_{ij}\right)_{n\times n}
 \end{eqnarray*}
 for some positive constant $k_{0}>0$,
$\beta_{\vec{\nu}}=D_{\vec{\nu}}\beta=-1$ and $|D\beta|=1$ on
$\partial\Omega$. Besides, since $\Omega$ is strictly convex, we
have
 \begin{eqnarray*}
\left(h_{ij}^{\partial\Omega}\right)_{(n-1)\times(n-1)}\geq
\kappa_{1}\left(\delta_{ij}\right)_{(n-1)\times(n-1)},
 \end{eqnarray*}
where $h_{ij}^{\partial\Omega}$, $1\leq i,j\leq n-1$, is the second
fundamental form of the boundary $\partial\Omega$, and
$\kappa_{1}>0$ is the minimal principal curvature of
$\partial\Omega$.

\begin{lemma}\label{lemma2.2}
Assume that $u(x,t)\in C^{3,2}(\Omega\times[0,T))$ is a solution to
the IBVP $(\sharp)$, and the Ricci curvature of $M^{n}$ is
nonnegative. Then there exists a constant
$c_{1}:=c_{1}(n,\Omega,u_{0},\phi(x))$ such that
\begin{eqnarray*}
\sup\limits_{\Omega\times[0,T)}|Du|\leq c_{1}.
\end{eqnarray*}
 \end{lemma}

\begin{proof}
To reach the conclusion of this lemma, we only need to prove that
for $0<T^{'}<T$, we can bound $|Du|$ on
$\overline{\Omega\times[0,T^{'})}$ independent of $T^{'}$ and then
take a limit argument.

Let
\begin{eqnarray*}
\Phi(x):=\log|D\omega|^{2}+f(\beta),
\end{eqnarray*}
where
\begin{eqnarray*}
\omega=u+\phi(x)\beta, \qquad f=\zeta \beta,
\end{eqnarray*}
and $\zeta$ is a positive constant which will be determined later.
For convenience, denote by $G=-\phi(x)\beta$.

We firstly show that the maximum of $\Phi(x)$ on
$\overline{\Omega}\times[0,T^{'}]$ cannot be achieved at the
boundary $\partial\Omega\times[0,T^{'}]$.

Choose a suitable local coordinates around a point
$x_{0}\in\overline{\Omega}$ such that $\tau_{n}$ is the inward unit
normal vector of $\partial\Omega$, and $\tau_{i}$,
$i=1,2,\cdots,n-1$, are the unit smooth tangent vectors of
$\partial\Omega$. Denote by $D_{\tau_{i}}u:=u_{i}$,
$D_{\tau_{j}}u:=u_{j}$, $D_{i}D_{j}u:=u_{ij}$ for $1\leq i,j \leq
n$.\footnote{ Covariant derivatives of other tensors can be
simplified similarly. For instance, one has
$\omega_{i}=D_{\tau_i}\omega$,
$\omega_{ij}=D_{\tau_i}D_{\tau_j}\omega$.} By the boundary
condition, one has
 \begin{eqnarray*}
D_{\tau_{n}}\omega\big{|}_{\partial\Omega}=\omega_{n}\big{|}_{\partial\Omega}=u_{n}\big{|}_{\partial\Omega}+\left(\phi_{n}\beta+\beta_{n}\phi\right)\big{|}_{\partial\Omega}=0.
 \end{eqnarray*}
  If $\Phi(x,t)$ attains its maximum at
$(x_{0},t_{0})\in\partial\Omega\times[0,T^{'}],$ then at
$(x_{0},t_{0}),$ we have
\begin{eqnarray} \label{add-1}
0\geq\Phi_{n}&=&\frac{|D\omega|^{2}_{n}}{|D\omega|^{2}}-\zeta\nonumber\\
&=&
\sum\limits_{k=1}^{n-1}\frac{2\omega^{k}D_{\tau_{n}}D_{\tau_{k}}\omega}{|D\omega|^{2}}-\zeta\nonumber\\
&=&
\sum\limits_{k=1}^{n-1}\frac{2\omega^{k}[\tau_{k}(\tau_{n}(\omega))-(D_{\tau_{k}}\tau_{n})\omega]}{|D\omega|^{2}}-\zeta\nonumber\\
&=&
-\sum\limits_{k=1}^{n-1}\frac{2\omega^{k}(D_{\tau_{k}}\tau_{n})(\omega)}{|D\omega|^{2}}-\zeta\nonumber\\
&=&
-\sum\limits_{k=1}^{n-1}\frac{2\omega^{k}\omega_{j}\langle D_{\tau_{k}}\tau_{n},\tau_{j}\rangle_{\sigma}}{|D\omega|^{2}}-\zeta\nonumber\\
&=&
\sum\limits_{k=1}^{n-1}\frac{2\omega^{k}\omega_{j}\langle D_{\tau_{k}}\tau_{j},\tau_{n}\rangle_{\sigma}}{|D\omega|^{2}}-\zeta\nonumber\\
&=&
\sum\limits_{k,j=1}^{n-1}\frac{2\omega^{k}\omega_{j}h_{kj}^{\partial\Omega}}{|D\omega|^{2}}-\zeta\nonumber\\
&\geq& 2\kappa_{1}-\zeta.
\end{eqnarray}
Hence, by taking $0<\zeta<2\kappa_{1}$, the maximum of $\Phi$ can
only be achieved in $\Omega\times[0,T^{'}]$. BTW, there is one thing
we would like to mention here, that is , in (\ref{add-1}), the
relation
 \begin{eqnarray*}
w^{k}=\sum\limits_{l=1}^{n}\sigma^{kl}w_{l}=\sum\limits_{l=1}^{n-1}\sigma^{kl}w_{l}
\end{eqnarray*}
holds. Here we have used the convention in Riemannian Geometry to
deal with the subscripts and superscripts, and this convention will
also be  used in the sequel.

Assume that  $\Phi(x,t)$ attains its maximum at
$(x_{0},t_{0})\in\Omega\times[0,T^{'}]$. By direct calculation, we
have
\begin{eqnarray*}
\Phi_{t}(x_{0},t_{0})=\frac{|D\omega|^{2}_{t}}{|D\omega|^{2}},
\end{eqnarray*}
\begin{eqnarray} \label{2.7-1}
\Phi_{i}(x_{0},t_{0})=\frac{|D\omega|^{2}_{i}}{|D\omega|^{2}}+\zeta
\beta_{i}=0
\end{eqnarray}
and
\begin{eqnarray} \label{2.8-1}
\Phi_{ij}(x_{0},t_{0})&=&
\frac{|D\omega|^{2}_{ij}}{|D\omega|^{2}}-\frac{|D\omega|^{2}_{i}|D\omega|^{2}_{j}}{|D\omega|^{4}}+\zeta\beta_{ij}\nonumber\\
&=& \frac{|D\omega|^{2}_{ij}}{|D\omega|^{2}}+\zeta
\beta_{ij}-\zeta^{2}\beta_{i}\beta_{j}.
\end{eqnarray}
Since $g^{ij}=\sigma^{ij}-\frac{D^{i}uD^{j}u}{1+|Du|^{2}}$, we have
\begin{eqnarray} \label{2.9-1}
0&\geq&\sum\limits_{i,j=1}^{n}g^{ij}\Phi_{ij}-\Phi_{t}\nonumber\\
&=&\sum\limits_{i,j=1}^{n}g^{ij}\frac{|D\omega|^{2}_{ij}}{|D\omega|^{2}}-\frac{|D\omega|^{2}_{t}}{|D\omega|^{2}}
+\zeta\sum\limits_{i,j=1}^{n}g^{ij}\beta_{ij}-\zeta^{2}\sum\limits_{i,j=1}^{n}g^{ij}\beta_{i}\beta_{j}\nonumber\\
&\triangleq& I_{1}+I_{2},
\end{eqnarray}
where
\begin{eqnarray*}
I_{1}=\sum\limits_{i,j=1}^{n}g^{ij}\frac{|D\omega|^{2}_{ij}}{|D\omega|^{2}}-\frac{|D\omega|^{2}_{t}}{|D\omega|^{2}}
\end{eqnarray*}
and
\begin{eqnarray*}
I_{2}=\sum\limits_{i,j=1}^{n}(\zeta g^{ij}\beta_{ij}-
\zeta^{2}g^{ij}\beta_{i}\beta_{j}).
\end{eqnarray*}

At $(x_{0},t_{0})$, one can make a suitable change\footnote{ This
change can always be found. In fact, one can firstly rotate
$\tau_i$, $i=1,2,\cdots,n$, such that the gradient vector $Du$ lies
in the same or the opposite direction with $\tau_1$. Denote by the
hyperplane, which is orthogonal with $\tau_1$, by $\Pi$. Then rotate
$\tau_{2},\tau_{3},\cdots,\tau_{n}$ in $\Pi$, corresponding to an
orthogonal matrix, such that the real symmetric matrices
$(u_{ij})_{2\leq i,j\leq n}$, $(\sigma_{ij})_{2\leq i,j\leq n}$
change into diagonal matrices.} to the coordinate vector fields
$\{\tau_{1},\tau_{2},\cdots,\tau_{n}\}$ such that $|Du|=u_{1}$,
$(u_{ij})_{2\leq i,j\leq n}$ is diagonal, and $(\sigma_{ij})_{2\leq
i,j\leq n}$ is diagonal. Clearly, in this setting, $\sigma^{11}=1$.
Besides, we  Then
 \begin{eqnarray*}
g^{11}&=&\frac{1}{v^{2}}, \quad g^{ij}=0~\mathrm{for} ~2\leq i,j\leq
n,i\neq
j,~\mathrm{and}~g^{ii}=\sigma^{ii}~\mathrm{for} ~i\geq2,\\
(g^{ij})_{k}&=&\left(\sigma^{ij}-\frac{D^{i}uD^{j}u}{v^{2}}\right)_{k}\\
&=&-\left(\frac{2u_{k}^{i}u^{j}}{v^{2}}-\sum\limits_{m=1}^{n}\frac{2u^{m}u_{mk}u^{i}u^{j}}{v^{4}}\right),
 \end{eqnarray*}
where $v=\sqrt{1+|Du|^{2}}=\sqrt{1+u^{2}_{1}}$.

Assume that $u_{1}$ is big enough such that
$u_{1},\omega_{1},\omega^{1},|D\omega|$ and $v$ are equivalent with
each other at $(x_{0},t_{0})$. Otherwise, the conclusion of Lemma
\ref{lemma2.2} is proved. It's also noticeable that
$|\omega_{i}|\leq c_2$, $i=2,\cdots,n$, for some nonnegative
constant $c_{2}$. \emph{Here, in the proof, $c_2$ is denoted to be a
nonnegative constant which may changes in different places but has
nothing to do with $T^{'}$}. Since $\left(\beta_{ij}\right)_{n\times
n}\geq k_{0}\left(\delta_{ij}\right)_{n\times n}$, one can easily
get
\begin{eqnarray} \label{2.10-1}
I_{2}&=&\sum\limits_{i,j=1}^{n}(\zeta g^{ij}\beta_{ij}-\zeta^{2}g^{ij}\beta_{i}\beta_{j})\nonumber\\
&\geq&
\zeta\left[\sum\limits_{i=2}^{n}\sigma^{ii}k_{0}+\frac{k_{0}}{v^{2}}\right]-\zeta^{2}\left(\frac{\beta_{1}^{2}}{v^{2}}+\sum\limits_{i=2}^{n}\sigma^{ii}\beta_{i}^{2}\right).
\end{eqnarray}

Set
$J:=\sum\limits_{i,j=1}^{n}g^{ij}|D\omega|^{2}_{ij}-|D\omega|^{2}_{t}$,
one has
\begin{eqnarray*}
J&=&\sum\limits_{i,j=1}^{n}g^{ij}|D\omega|^{2}_{ij}-|D\omega|^{2}_{t}\\
&=&\sum\limits_{i,j,k=1}^{n}g^{ij}(2\omega^{k}_{j}\omega_{ki}+2\omega^{k}\omega_{ikj})-2\sum\limits_{k=1}^{n}\omega^{k}\omega_{tk}\\
&=&
\sum\limits_{i,j,k,l=1}^{n}g^{ij}[2\omega^{k}_{j}\omega_{ki}+2\omega^{k}(u_{ijk}+R^{l}_{ikj}u_{l}-G_{ikj})]-2\sum\limits_{k=1}^{n}\omega^{k}\omega_{tk}\\
&=&
\sum\limits_{i,j,k=1}^{n}g^{ij}[2\omega^{k}_{j}\omega_{ki}+2\omega^{k}(u_{ijk}+R^{1}_{ikj}u_{1}-G_{ikj})]-2\sum\limits_{k=1}^{n}\omega^{k}\omega_{tk}\\
&=&
2\sum\limits_{k=1}^{n}\omega^{k}\left[\sum\limits_{i,j=1}^{n}g^{ij}(u_{ijk}+R^{1}_{ikj}u_{1}-G_{ikj})-\omega_{tk}\right]
+2\sum\limits_{i,j,k=1}^{n}g^{ij}\omega_{ki}\omega^{k}_{j}\\
&=&
2\sum\limits_{k=1}^{n}\omega^{k}\left[\sum\limits_{i,j=1}^{n}g^{ij}(u_{ijk}+R^{1}_{ikj}u_{1}-G_{ijk})-u_{tk}\right]
+2\sum\limits_{i,j,k=1}^{n}g^{ij}\omega_{ki}\omega^{k}_{j}\\
&=&
-2\sum\limits_{i,j,k=1}^{n}g^{ij}\omega^{k}G_{ijk}+2\sum\limits_{i,j,k=1}^{n}g^{ij}\omega^{k}R^{1}_{ikj}u_{1}
-2\sum\limits_{i,j,k=1}^{n}\omega^{k}(g^{ij})_{k}u_{ij}+2\sum\limits_{i,j,k=1}^{n}g^{ij}\omega_{ki}\omega^{k}_{j}\\
&\triangleq& J_{1}+J_{2}+J_{3}+J_{4},
\end{eqnarray*}
where $R^{l}_{ikj}$, $1\leq i,j,k,l\leq n$, are coefficients of the
curvature tensor on $M^n$. It is obvious that
\begin{eqnarray} \label{2.11-1}
J_{1}&=&-2\sum\limits_{i,j,k=1}^{n}g^{ij}\omega^{k}G_{ijk}\nonumber\\
&\geq&-c_{2}v.
\end{eqnarray}
Next we deal with $J_{2},J_{3}$. In fact, using the nonnegativity of
the Ricci curvature on $M^n$, we have
\begin{eqnarray}\label{2.12-1}
J_{2}&=&2\sum\limits_{i,j,k=1}^{n}g^{ij}\omega^{k}R^{1}_{ikj}u_{1}\nonumber\\
&=&2\sum\limits_{i=2}^{n}\sigma^{ii}R^{1}_{i1i}v^{2}+2\sum\limits_{i,k=2}^{n}\sigma^{ii}\omega^{k}R^{1}_{iki}v+2R^{1}_{111}
+2\sum\limits_{k=2}^{n}\frac{\omega^{k}R^{1}_{1k1}}{v}\nonumber\\
&\geq& c_{2}v^{2}+O(1)v+2R^{1}_{111}\nonumber\\
&\geq& c_{2}v^{2}+O(1)v,
\end{eqnarray}
 and
\begin{eqnarray*}
J_{3}&=&-2\sum\limits_{i,j,k=1}^{n}\omega^{k}(g^{ij})_{k}u_{ij}\\
&=&\sum\limits_{i,j,k=1}^{n}4\omega^{k}\frac{u_{k}^{i}u^{j}u_{ij}}{v^{2}}-\sum\limits_{m,i,j,k=1}^{n}4\omega^{k}\frac{u^{m}u^{i}u^{j}u_{mk}u_{ij}}{v^{4}}\\
&=&\sum\limits_{i,l,k=1}^{n}4\omega^{k}\frac{\sigma^{il}u_{1}u_{lk}u_{1i}}{v^{2}}
-\sum\limits_{k=1}^{n}4\omega^{k}\frac{u^{3}_{1}u_{1k}u_{11}}{v^{4}}\\
&=&4\frac{\omega^{1}(u_{11})^{2}u_{1}}{v^{4}}+4\sum\limits_{i=2}^{n}\frac{\omega^{i}u_{1}u_{11}u_{1i}}{v^{4}}
+4\sum\limits_{i=2}^{n}\frac{\omega^{1}u_{1}\sigma^{ii}(u_{1i})^{2}}{v^{2}}+4\sum\limits_{i=2}^{n}\frac{\omega^{i}\sigma^{ii}u_{1i}u_{ii}u_{1}}{v^{2}}\\
&\triangleq& J_{31}+J_{32}+J_{33}+J_{34}.
\end{eqnarray*}
By (\ref{2.7-1}) and (\ref{2.8-1}), for $2\leq i\leq n$, we have
\begin{eqnarray} \label{2.13-1}
u_{1i}-G_{1i}=\frac{-\zeta\beta_{i}|D\omega|^{2}-2\omega^{i}u_{ii}+2\sum\limits_{k=2}^{n}\omega^{k}G_{ik}}{2\omega^{1}}
\end{eqnarray}
and
\begin{eqnarray} \label{2.13-2}
\sum\limits_{i=2}^{n}\frac{2\omega^{i}(u_{1i}-G_{1i})}{|D\omega|^{2}}=-\zeta\beta_{1}-\frac{2\omega^{1}(u_{11}-G_{11})}{|D\omega|^{2}}.
\end{eqnarray}
By (\ref{2.13-1}), for $2\leq i\leq n$, it follows that
\begin{eqnarray} \label{2.14-1}
u_{1i}&=&\frac{-\zeta\beta_{i}|D\omega|^{2}-2\omega^{i}u_{ii}+2\sum\limits_{k=2}^{n}\omega^{k}G_{ik}}{2\omega^{1}}+G_{1i}\nonumber\\
&=&
-\frac{1}{2}\zeta\beta_{i}v-\frac{\omega^{i}u_{ii}}{\omega^{1}}+O(1).
\end{eqnarray}
By (\ref{2.13-2}), we have
\begin{eqnarray} \label{2.15-1}
\sum\limits_{i=2}^{n}\frac{2\omega^{i}(u_{1i}-G_{1i})}{|D\omega|^{2}}
&=&\sum\limits_{i=2}^{n}\frac{2\omega^{i}}{|D\omega|^{2}}\left(\frac{-\zeta\beta_{i}|D\omega|^{2}-2\omega^{i}u_{ii}
+2\sum\limits_{k=2}^{n}\omega^{k}G_{ik}}{2\omega^{1}}\right)\nonumber\\
&=& \sum\limits_{i=2}^{n}O(\frac{|\zeta\beta_{i}|}{v})
-\sum\limits_{i=2}^{n}\frac{2(\omega^{i})^{2}u_{ii}}{|D\omega|^{2}\omega^{1}}+\sum\limits_{i,k=2}^{n}\frac{2\omega^{i}\omega^{k}G_{ki}}{|D\omega|^{2}\omega^{1}}.
\end{eqnarray}
By (\ref{2.14-1}) and (\ref{2.15-1}) , we have
\begin{eqnarray*}
-\zeta\beta_{1}-\frac{2\omega^{1}(u_{11}-G_{11})}{|D\omega|^{2}}&=&
\sum\limits_{i=2}^{n}O(\frac{|\zeta\beta_{i}|}{v})-\sum\limits_{i=2}^{n}\frac{2(\omega^{i})^{2}u_{ii}}{|D\omega|^{2}\omega^{1}}
+\sum\limits_{i,k=2}^{n}\frac{2\omega^{i}\omega^{k}G_{ki}}{|D\omega|^{2}\omega^{1}}\\
-\zeta\beta_{1}-\frac{2v(u_{11}-G_{11})}{v^{2}}&=&
O(\frac{1}{v})+O(\frac{1}{v^{3}})u_{ii}+O(\frac{1}{v^{3}}).
\end{eqnarray*}
So,
\begin{eqnarray} \label{2.16-1}
u_{11}=-\frac{1}{2}\zeta\beta_{1}v+\sum\limits_{i=2}^{n}O(\frac{1}{v^{2}})u_{ii}+O(1)
\end{eqnarray}

Now, we deal with $J_{31}, J_{32}, J_{33}, J_{34}$ respectively. It
is obvious that
\begin{eqnarray} \label{2.17-1}
J_{31}+J_{33}\geq0.
\end{eqnarray}
For the term $J_{32}$,
\begin{eqnarray} \label{2.18-1}
J_{32}&=&4\sum\limits_{i=2}^{n}\frac{\omega^{i}u_{1}u_{11}u_{1i}}{v^{4}}\nonumber\\
&=&4\sum\limits_{i=2}^{n}\frac{\omega^{i}u_{1}}{v^{4}}\left(-\frac{1}{2}\zeta\beta_{1}v+\sum\limits_{i=2}^{n}O(\frac{1}{v^{2}})u_{ii}+O(1)\right)\cdot
\left(-\frac{1}{2}\zeta\beta_{i}v-\frac{\omega^{i}u_{ii}}{\omega^{1}}+O(1)\right)\nonumber\\
&=&O(\frac{\zeta^{2}|\beta_{i}\beta_{1}|}{v})+O(\frac{1}{v^{2}})
+\sum\limits_{i=2}^{n}\left[O(\frac{|\zeta\beta_{1}|}{v^{2}})+O(\frac{|\zeta\beta_{i}|}{v^{4}})\right]u_{ii}+\sum\limits_{i=2}^{n}O(\frac{1}{v^{6}})u_{ii}^{2}.
\end{eqnarray}
Besides, we have
\begin{eqnarray} \label{2.19-1}
J_{34}&=&4\sum\limits_{i=2}^{n}\frac{\omega^{i}\sigma^{ii}u_{1i}u_{ii}u_{1}}{v^{2}}\nonumber\\
&=&
4\sum\limits_{i=2}^{n}\frac{\omega^{i}u_{1}\sigma^{ii}u_{ii}}{v^{2}}\left(-\frac{1}{2}\zeta\beta_{i}v-\frac{\omega^{i}u_{ii}}{\omega^{1}}+O(1)\right)\nonumber\\
&=&
\sum\limits_{i=2}^{n}O(\frac{1}{v^{2}})u_{ii}^{2}+\sum\limits_{i=2}^{n}O(|\zeta\beta_{i}|)u_{ii}.
\end{eqnarray}
By (\ref{2.17-1})-(\ref{2.19-1}), it follows that
\begin{eqnarray} \label{2.20-1}
J_{3}&\geq&\sum\limits_{i=2}^{n}\left[O(\frac{1}{v^{6}})+O(\frac{1}{v^{2}})\right]u_{ii}^{2}
+\sum\limits_{i=2}^{n}\left[O(\frac{|\zeta\beta_{1}|}{v^{2}})+O(\frac{|\zeta\beta_{i}|}{v^{4}})+O(|\zeta\beta_{i}|)\right]u_{ii}\nonumber\\
&&+O(\frac{\zeta^{2}|\beta_{i}\beta_{1}|}{v})+O(\frac{1}{v^{2}}).
\end{eqnarray}
Then, for $J_{4}$, we have
\begin{eqnarray} \label{2.23-1}
J_{4}&=&2\sum\limits_{i,j,k=1}^{n}g^{ij}\omega_{ki}\omega^{k}_{j}\nonumber\\
&=&
2\sum\limits_{l,i,j,k=1}^{n}g^{ij}\sigma^{kl}\omega_{ki}\omega_{lj}\nonumber\\
&=&
2\sum\limits_{i=1}^{n}\frac{\sigma^{ii}}{v^{2}}(\omega_{1i})^{2}+2\sum\limits_{i=2}^{n}\sigma^{ii}(\omega_{1i})^{2}
+2\sum\limits_{i=2}^{n}(\sigma^{ii})^{2}(\omega_{ii})^{2}\nonumber\\
&\geq&\sum\limits_{i=2}^{n}(1+\frac{1}{v^{2}})\sigma^{ii}u_{1i}^{2}+\sum\limits_{i=2}^{n}(\sigma^{ii})^{2}u_{ii}^{2}-c_{2}.
\end{eqnarray}
By (\ref{2.11-1}), (\ref{2.12-1}), (\ref{2.20-1}) and
(\ref{2.23-1}), we write all the terms containing $u_{ii}$ in $J$ as
below
\begin{eqnarray*}
&&\sum\limits_{i=2}^{n}\left[O(\frac{1}{v^{6}})+O(\frac{1}{v^{2}})+(\sigma^{ii})^{2}\right]u_{ii}^{2}+\sum\limits_{i=2}^{n}\left[O(|\zeta\beta_{i}|)
+O(\frac{|\zeta\beta_{i}|}{v^{4}})+O(\frac{|\zeta\beta_{1}|}{v^{2}})\right]u_{ii}\nonumber\\
&&\qquad
\geq-\sum\limits_{i=2}^{n}\frac{O(|\zeta\beta_{i}|^{2})}{(\sigma^{ii})^{2}},
\end{eqnarray*}
where the inequality holds since $ax^{2}+bx\geq-\frac{b^{2}}{4a}$
for $a>0$. Therefore, we can obtain
\begin{eqnarray} \label{2.23-1}
J&=&J_{1}+J_{2}+J_{3}+J_{4}\nonumber\\
&\geq&-\sum\limits_{i=2}^{n}\frac{O(|\zeta\beta_{i}|^{2})}{(\sigma^{ii})^{2}}
-c_{2}v+O(1)v+c_{2}v^{2}.
\end{eqnarray}
 Hence,
\begin{eqnarray} \label{2.24-1}
I_{1}&=&\frac{J}{|D\omega|^{2}}\nonumber\\
&\geq&
-\frac{\sum\limits_{i=2}^{n}\frac{O(|\zeta\beta_{i}|^{2})}{(\sigma^{ii})^{2}}+c_{2}v-O(1)v-c_{2}v^{2}}{|D\omega|^{2}}.
\end{eqnarray}
By (\ref{2.9-1}), (\ref{2.10-1}) and (\ref{2.24-1}), at the maximum
point $(x_{0},t_{0})$, we can get
\begin{eqnarray*}
0&\geq&\sum\limits_{i,j=1}^{n}g^{ij}\Phi_{ij}-\Phi_{t}\\
&\geq&-\frac{\sum\limits_{i=2}^{n}\frac{O(|\zeta\beta_{i}|^{2})}{(\sigma^{ii})^{2}}+c_{2}v-O(1)v-c_{2}v^{2}}{|D\omega|^{2}}\\
&&+\zeta\left[\sum\limits_{i=2}^{n}\sigma^{ii}k_{0}+\frac{k_{0}}{v^{2}}\right]-\zeta^{2}\left(\frac{\beta_{1}^{2}}{v^{2}}+\sum\limits_{i=2}^{n}\sigma^{ii}\beta_{i}^{2}\right)\\
&\geq&\zeta\left[\sum\limits_{i=2}^{n}\sigma^{ii}k_{0}+\frac{k_{0}}{v^{2}}\right]-\zeta^{2}\left(\frac{\beta_{1}^{2}}{v^{2}}+\sum\limits_{i=2}^{n}\sigma^{ii}\beta_{i}^{2}\right)
\end{eqnarray*}
Let $\lambda=\min(\sigma^{ii})$, $\Lambda=\max(\sigma^{ii})$,
$i\geq2$. Taking
$0<\zeta<\min\{\frac{\lambda(n-1)k_{0}}{{\Lambda}},2\kappa_{1}\}$,
we can obtain
\begin{eqnarray*}
v(x_{0},t_{0})\leq c_{3},
\end{eqnarray*}
where $c_{3}$ is independent of $T^{'}$. Then the conclusion of
Lemma \ref{lemma2.2} follows immediately.
\end{proof}

By Lemmas \ref{lemma2.1} and \ref{lemma2.2}, together with the
Schauder estimate for parabolic PDEs, we can get uniform estimates
in any $C^{k}$-norm for the derivatives of $u$, and locally (in
time) uniform bounds for the $C^{0}$-norm, which leads to the
long-time existence, with uniform bounds on all higher derivatives
of $u$, to the IBVP ($\sharp$). This finishes the proof of (1) of
Theorem \ref{theorem1-1}.

\section{Asymptotic behavior}
\renewcommand{\thesection}{\arabic{section}}
\renewcommand{\theequation}{\thesection.\arabic{equation}}
\setcounter{equation}{0} \setcounter{maintheorem}{0}

In order to study the asymptotic behavior of the solution to the
IVBP $(\sharp)$, we need the following two conclusions.

\begin{lemma}\label{lemma3.1}
Let $\Omega$ be a strictly convex bounded domain in $M^{n}$,
$n\geq2$, and $\partial\Omega\in C^{3}$. Assume that
$\varepsilon>0$, the Ricci curvature of $M^{n}$ is nonnegative,
$\phi$ is a function defined on $\overline{\Omega}$, and there
exists a positive constant $L>0$ such that
\begin{eqnarray*}
|\phi|_{C^{3}(\overline{\Omega})}\leq L
\end{eqnarray*}
Let $u\in C^{2}(\overline{\Omega})\cap C^{3}(\Omega)$ be a solution
to the following BVP
\begin{eqnarray} \label{3.2-add}
\left\{
\begin{array}{lll}
{\varepsilon u=\sum\limits_{i=1}^{n}\left(\sigma^{ij}-\frac{D^{i}uD^{j}u}{1+|Du|^{2}}\right)D_{i}D_{j}u} \qquad \qquad &\mathrm{in} ~ \Omega,\\
D_{\vec{\nu}}u=\phi(x) \qquad \qquad &\mathrm{on} ~
\partial\Omega,\\
\end{array}
\right.
\end{eqnarray}
then there exists a constant $c_{4}:=c_{4}(n,\Omega,L)>0$ such that
\begin{eqnarray*}
\sup\limits_{\overline{\Omega}}|Du|\leq c_{4}.
\end{eqnarray*}
 \end{lemma}

\begin{proof}
Let $\Phi(x)=\log|D\omega|^{2}+\zeta\beta$, where
$\omega=u+\phi(x)\beta$, and $\zeta$ will be determined later.
Denote by $G=-\phi(x)\beta$.

If one chooses $0<\zeta<2\kappa_{1}$, using an almost same procedure
as that in (\ref{add-1}), it is easy to show that the maximum of
$\Phi$ can only be achieved in the interior of $\Omega$.

Assume that $\Phi(x)$ attains its maximum at $x_{0}\in\Omega$, then
we have at this point that
\begin{eqnarray*}
\Phi_{i}(x_{0})=\frac{|D\omega|^{2}_{i}}{|D\omega|^{2}}+\zeta\beta_{i}=0
\end{eqnarray*}
and
\begin{eqnarray*}
0\geq\Phi_{ij}(x_{0}) &=&
\frac{|D\omega|^{2}_{ij}}{|D\omega|^{2}}-\frac{|D\omega|^{2}_{i}|D\omega|^{2}_{j}}{|D\omega|^{4}}+\zeta\beta_{ij}\nonumber\\
&=& \frac{|D\omega|^{2}_{ij}}{|D\omega|^{2}}+\zeta
\beta_{ij}-\zeta^{2}\beta_{i}\beta_{j}.
\end{eqnarray*}
It follows that
\begin{eqnarray} \label{3.5}
0&\geq&\sum\limits_{i,j=1}^{n}g^{ij}\Phi_{ij}\nonumber\\
&=&
\sum\limits_{i,j=1}^{n}g^{ij}\frac{|D\omega|^{2}_{ij}}{|D\omega|^{2}}+\sum\limits_{i,j=1}^{n}\zeta g^{ij}\beta_{ij}-\sum\limits_{i,j=1}^{n}\zeta^{2}g^{ij}\beta_{i}\beta_{j}\nonumber\\
&\triangleq& I_{1}+I_{2}
\end{eqnarray}
where
\begin{eqnarray*}
I_{1}=\sum\limits_{i,j=1}^{n}g^{ij}\frac{|D\omega|^{2}_{ij}}{|D\omega|^{2}}
\end{eqnarray*}
and
\begin{eqnarray*}
I_{2}=\sum\limits_{i,j=1}^{n}\zeta
g^{ij}\beta_{ij}-\sum\limits_{i,j=1}^{n}\zeta^{2}g^{ij}\beta_{i}\beta_{j}.
\end{eqnarray*}
As in Lemma \ref{lemma2.2}, one can choose suitable local
coordinates around $x_0$ such that $|Du|=u_{1}$, $(u_{ij})_{2\leq
i,j\leq n}$ is diagonal, and $(\sigma_{ij})_{2\leq i,j\leq n}$ is
diagonal. Similarly, for the term $I_{2}$, at $x_0$, we have
\begin{eqnarray*}
I_{2}&\geq
\zeta\left[\sum\limits_{i=2}^{n}\sigma^{ii}k_{0}+\frac{k_{0}}{v^{2}}\right]-\zeta^{2}\left(\frac{\beta_{1}^{2}}{v^{2}}+\sum\limits_{i=2}^{n}\sigma^{ii}\beta_{i}^{2}\right).
\end{eqnarray*}

Set $J:=\sum\limits_{i,j=1}^{n}g^{ij}|D\omega|^{2}_{ij}$. By direct
calculation, one has
\begin{eqnarray*}
\begin{split}
J=&\sum\limits_{i,j=1}^{n}g^{ij}|D\omega|^{2}_{ij}\\
=&\sum\limits_{i,j,k=1}^{n}g^{ij}(2\omega^{k}_{j}\omega_{ki}+2\omega^{k}\omega_{ikj})\\
=&
\sum\limits_{i,j,k,l=1}^{n}g^{ij}\left[2\omega^{k}_{j}\omega_{ki}+2\omega^{k}(u_{ijk}+R^{l}_{ikj}u_{l}-G_{ikj})\right]\\
=&
2\sum\limits_{i,j,k=1}^{n}g^{ij}\omega^{k}(u_{ijk}+R^{1}_{ikj}u_{1}-G_{ikj})+2\sum\limits_{i,j,k=1}^{n}g^{ij}\omega^{k}_{j}\omega_{ki}\\
=&
-2\sum\limits_{i,j,k=1}^{n}g^{ij}\omega^{k}G_{ijk}+2\sum\limits_{i,j,k=1}^{n}g^{ij}\omega^{k}R^{1}_{ikj}u_{1}-2\sum\limits_{i,j,k=1}^{n}\omega^{k}(g^{ij})_{k}u_{ij}\\
&+2\sum\limits_{k=1}^{n}\omega^{k}(\varepsilon u_{k})+2\sum\limits_{i,j,k=1}^{n}g^{ij}\omega_{ki}\omega^{k}_{j}\\
\triangleq& J_{1}+J_{2}+J_{3}+J_{4}+J_{5}.
\end{split}
\end{eqnarray*}
Without loss of generality, one may assume that $u_{1}$ is big
enough, then
\begin{eqnarray*}
J_{4}=2\varepsilon v^{2}\geq0.
\end{eqnarray*}
Otherwise, the conclusion of Lemma \ref{lemma3.1} follows.

As in Lemma \ref{lemma2.2}, the other terms $J_{1}$, $J_{2}$,
$J_{3}$, $J_{5}$ can be controlled similarly. In (\ref{3.5}), taking
$0<\zeta<\min\{\frac{\lambda(n-1)k_{0}}{{\Lambda}},2\kappa_{1}\}$,
we can obtain
 \begin{eqnarray*}
   \label{3.7} v(x_{0})\leq c_{5}
\end{eqnarray*}
for some positive constant $c_{5}:=c_{5}(n,\Omega,L)$, which is
independent of $\varepsilon$. Then the assertion of Lemma
\ref{lemma3.1} follows.
\end{proof}

\begin{theorem}\label{main3.2}
Let $\Omega$ be a strictly convex bounded domain in $M^{n}$ with
$C^{3}$ boundary $\partial\Omega$, $n\geq2$. Assume that the Ricci
curvature of $M^{n}$ is nonnegative. For $\phi(x)\in
C^{3}(\overline{\Omega})$, there exists a unique
$\lambda\in\mathbb{R}$ and $\omega\in
C^{2,\alpha}(\overline{\Omega})$ solving the BVP ($\ddag$)
 Moreover, the solution $\omega$
is unique up to a constant.
 \end{theorem}

\begin{proof}
We use a similar method to that of the proof of \cite[Theorem
1.2]{mww}.

 For each fixed $\varepsilon>0$, we firstly show the
existence of the solution to the BVP (\ref{3.2-add}). Based on the
$C^{1}$-estimate (see Lemma \ref{lemma3.1}), the only obstacle is to
derive a priori $C^{0}$-estimate for the solution
$u_{\varepsilon}(x)$ to the BVP (\ref{3.2-add}).

Let $f$ be a smooth function on $\overline{\Omega}$ satisfying
$D_{\vec{\nu}}f<-\sup_{\overline{\Omega}}|\phi(x)|$. Let $\rho$ be a
point where $f-u_{\varepsilon}$ achieves its minimum. Denote by $T$
the tangent vector to $\partial\Omega$. If $\rho\in\partial\Omega$,
then $D_{T}f(\rho)=D_{T}u_{\varepsilon}(\rho)$ and
$D_{\vec{\nu}}f(\rho)\geq
D_{\vec{\nu}}u_{\varepsilon}(\rho)=\phi(\rho)$, which is contradict
with the choice of $f$. So, $\rho\in\Omega$, and then
~$Df(\rho)=Du_{\varepsilon}(\rho)$ and $D^{2}f(\rho)\geq
D^{2}u_{\varepsilon}(\rho)$. This gives the existence of a constant
$c_{6}:=c_{6}(f)$ such that
 \begin{eqnarray*}
c_{6}\geq\sum\limits_{i,j=1}^{n}
g^{ij}(Df)f_{ij}(\rho)\geq\sum\limits_{i,j=1}^{n}
g^{ij}(Du_{\varepsilon})(u_{\varepsilon})_{ij}(\rho)=\varepsilon
u_{\varepsilon}(\rho)
 \end{eqnarray*}
Together with the fact $f(x)-u_{\varepsilon}(x)\geq
f(\rho)-u_{\varepsilon}(\rho)$ for $x\in\Omega$, we have
 \begin{eqnarray*}
\varepsilon u_{\varepsilon}(x)\leq\varepsilon f(x)-\varepsilon
f(\rho)+c_{6}.
 \end{eqnarray*}
 Similarly, one can get a lower bound for $\varepsilon
 u_{\varepsilon}(x)$. Therefore, $\sup_{\overline{\Omega}}|\varepsilon u_{\varepsilon}|\leq
 c_{7}$ holds for some nonnegative constant $c_{7}$. By the standard
 theory of second-order elliptic PDEs, one can get the existence of the solution
 to the BVP (\ref{3.2-add}).

Set
$\omega_{\varepsilon}:=u_{\varepsilon}-\frac{1}{|\Omega|}\int_{\Omega}u_{\varepsilon}dx$.
It is easy to check that $\omega_{\varepsilon}$ satisfies
\begin{eqnarray*}
 \left\{
\begin{array}{lll}
\sum\limits_{i,j=1}^{n}\left(\sigma^{ij}-\frac{(\omega_{\varepsilon})^{i}(\omega_{\varepsilon})^{j}}{1+|D\omega_{\varepsilon}|^{2}}\right)(\omega_{\varepsilon})_{ij}
=\varepsilon\omega_{\varepsilon}+\varepsilon\frac{1}{|\Omega|}\int\limits_{\Omega}u_{\varepsilon}dx \qquad \qquad &\mathrm{in} ~ \Omega,\\
D_{\vec{\nu}}\omega_{\varepsilon}=(\omega_{\varepsilon})_{\vec{\nu}}=\phi(x)
\qquad \qquad &\mathrm{on} ~
\partial \Omega.
\end{array}
\right.
\end{eqnarray*}

By
 \begin{eqnarray*}
 \sup\limits_{\overline{\Omega}}|D\omega_{\varepsilon}|=\sup\limits_{\overline{\Omega}}|Du_{\varepsilon}|\leq
 c_{4}
 \end{eqnarray*}
 (see Lemma \ref{lemma3.1})  and the fact that $\omega_{\varepsilon}$ has at least one zero
point, we have $|\omega_{\varepsilon}|\leq c_{8}$ for some
nonnegative constant $c_{8}:=c_{8}(c_{4},c_{7})$, which also gives
the boundedness of $\frac{1}{|\Omega|}\int_{\Omega} (\varepsilon
u_{\varepsilon})dx$. By the Schauder theory for second-order
elliptic PDEs, one has
$|\omega_{\varepsilon}|_{C^{2,\alpha}(\overline{\Omega})}\leq c_{9}$
for some nonnegative constant $c_{9}:=c_{9}(c_8)$. Taking
$\varepsilon\rightarrow0$, we have
$\omega_{\varepsilon}\rightarrow\omega$ and
$\varepsilon\omega_{\varepsilon}+\varepsilon\frac{1}{|\Omega|}\int_{\Omega}u_{\varepsilon}dx\rightarrow\lambda$,
where $(\lambda,\omega)$ solves the BVP ($\ddag$).

Assume that there exist two pairs $(\lambda_{1},u_{1})$~ and~
$(\lambda_{2},u_{2})$ solving the BVP ($\ddag$). Without loss of
generality, we may assume that $\lambda_{1}\leq \lambda_{2}$. Let
$\omega=u_{1}-u_{2}$, and by the linearization process for the
quasilinear elliptic PDEs, it is clear that $\omega$ satisfies
\begin{eqnarray}
 \left\{
\begin{array}{lll}
\sum\limits_{i,j=1}^{n}\widetilde{g}^{ij}\omega_{ij}+\sum\limits_{i}^{n}b_{i}\omega_{i}=\lambda_{1}-\lambda_{2}\leq0 ~ \qquad \qquad &\mathrm{in} ~ \Omega,\\
D_{\vec{\nu}}\omega=0 \qquad \qquad &\mathrm{on} ~
\partial \Omega,
\end{array}
\right.
\end{eqnarray}
where $\widetilde{g}^{ij}=g^{ij}(Du_{1})$ and ~
$b_{i}=\sum\limits_{k,l=1}^{n}(u_{2})_{kl}\int_{0}^{1}g^{kl}_{,p_{i}}(\eta
Du_{1}+(1-\eta)Du_{2})d\eta$. By Hopf's lemma, $\omega$ must be a
constant, which gives the uniqueness (up to a constant) of the
solution to the BVP ($\ddag$). Consequently, we have
$\lambda_{1}=\lambda_{2}$. This completes the proof of Theorem
\ref{main3.2}.
\end{proof}

Let
\begin{eqnarray}
\begin{array}{lll}
\widetilde{\omega}(x,t):=\omega+\lambda t,
\end{array}
\end{eqnarray}
where $(\lambda,\omega)$ is the solution to the BVP ($\ddag$). It's
easy to check that $\widetilde{\omega}$ solves the following IBVP
\begin{eqnarray}
\left\{
\begin{array}{lll}
{u_{t}=\sum\limits_{i,j=1}^{n}\left(\sigma^{ij}-\frac{D^{i}uD^{j}u}{1+|Du|^{2}}\right)D_{i}D_{j}u} \qquad \qquad &\mathrm{on} ~ \Omega\times(0,\infty),\\
D_{\vec{\nu}}u=\phi(x) \qquad \qquad &\mathrm{on} ~
\partial \Omega\times(0,\infty),\\
u(x,0)=\omega(x)\qquad \qquad &\mathrm{on} ~\Omega.
\end{array}
\right.
\end{eqnarray}
As mentioned at the end of Section 2, by Lemmas \ref{lemma2.1} and
\ref{lemma2.2}, the Schauder theory for parabolic PDEs, one can
obtain the long-time existence for the IBVP  $(\sharp)$, i.e.,
$T=\infty$.

\begin{corollary}\label{main3.3}
For a solution $u=u(x,t)$ to the IBVP  $(\sharp)$, there exists a
positive constant $c_{10}$, independent of $t$, such that
\begin{eqnarray*}
|u(x,t)-\lambda t|\leq c_{10}.
\end{eqnarray*}
\end{corollary}

\begin{proof}
Set $z(x,t):=u(x,t)-\widetilde{\omega}(x,t)$. By the linearization
process, it is easy to check that $z(x,t)$ satisfies
\begin{eqnarray*}
 \left\{
\begin{array}{lll}
z_{t}=\sum\limits_{i,j=1}^{n}\widetilde{g}^{ij}z_{ij}+\sum\limits_{i=1}^{n}b_{i}z_{i} ~ \qquad \qquad &\mathrm{in} ~ \Omega\times(0,\infty),\\
D_{\vec{\nu}}z=0 \qquad \qquad &\mathrm{on} ~
\partial \Omega\times(0,\infty),\\
z(x,0)=u_{0}(x)-\omega(x)~ \qquad \qquad &\mathrm{on} ~ \Omega,
\end{array}
\right.
\end{eqnarray*}
where $\widetilde{g}^{ij}=g^{ij}(Du)$ and
$b_{i}=\sum\limits_{k,l=1}^{n}(\widetilde{\omega})_{kl}\int_{0}^{1}g^{kl}_{,p_{i}}(\eta
Du+(1-\eta)D\widetilde{\omega})d\eta$. By the maximum principle of
second-order parabolic PDEs, we know that $z$ attains its maximum
and minimum on $\Omega\times\{0\}$. Hence, one has
\begin{eqnarray*}
\sup\limits_{\Omega\times(0,\infty)}|u-\lambda t|\leq
\sup\limits_{\Omega}|\omega|+\sup\limits_{\Omega}|u_{0}-\omega|,
\end{eqnarray*}
which implies the conclusion of Corollary \ref{main3.3}.
\end{proof}

\begin{lemma}\label{main3.4}
Let $u_{1}$ and $u_{2}$ be any two solutions to the IBVP $(\sharp)$
with initial data $u_{0,1}$~and~$u_{0,2}$ respectively. Let
$u=u_{1}-u_{2}$, then $u$ converges to a constant function as $t
\rightarrow\infty$. In particular, the limit of any solution to the
IBVP $(\sharp)$ is $\widetilde{\omega}$ up to a constant.
\end{lemma}

\begin{proof}
We use a similar method to that of the proof of \cite[Lemma
2.5]{mww}.

 As shown in Corollary \ref{main3.3}, it is easy to know that $u$
 satisfies
\begin{eqnarray}
 \left\{
\begin{array}{lll}
z_{t}=\sum\limits_{i,j=1}^{n}\widetilde{g}^{ij}z_{ij}+\sum\limits_{i}^{n}b_{i}z_{i} ~ \qquad \qquad &\mathrm{in} ~ \Omega\times(0,\infty),\\
z_{\nu}=0 \qquad \qquad &\mathrm{on} ~
\partial \Omega\times(0,\infty),\\
z(x,0)=u_{0,1}(x)-u_{0,2}(x)~ \qquad \qquad &\mathrm{on} ~ \Omega,
\end{array}
\right.
\end{eqnarray}
where $\widetilde{g}^{ij}=g^{ij}(Du_{1})$ and
$b_{i}=\sum\limits_{k,l=1}^{n}(u_{2})_{kl}\int_{0}^{1}g^{kl}_{,p_{i}}(\eta
Du_{1}+(1-\eta)Du_{2})d\eta$. Set
\begin{eqnarray*}
\mathrm{osc}(u)(t)=\max\limits_{\Omega}u(x,t)-\min\limits_{\Omega}u(x,t).
\end{eqnarray*}
By the strong maximum principle of second-order parabolic PDEs and
Hopf's lemma, one knows that $\mathrm{osc}(u)(t)$ is a strictly
decreasing function unless $u$ is a constant.

Now, we \textbf{claim} that
\begin{eqnarray*}
\lim\limits_{t\rightarrow\infty}\mathrm{osc}(u)(t)=0.
\end{eqnarray*}
Otherwise, one has
$\lim\limits_{t\rightarrow\infty}\mathrm{osc}(u)(t)=\chi$ for some
 $\chi>0$. In fact, given a sequence
$t_{n}\rightarrow+\infty$, define
\begin{eqnarray*}
u_{1,n}(\cdot,t):=u_{1}(\cdot,t+t_{n})-\lambda t_{n}
\end{eqnarray*}
and
\begin{eqnarray*}
u_{2,n}(\cdot,t):=u_{2}(\cdot,t+t_{n})-\lambda t_{n}
\end{eqnarray*}
By Corollary \ref{main3.3}, for $i=1,2$, we have $|u_{i,n}-\lambda
t|\leq c_{10}$. By Lemmas \ref{lemma2.1}, \ref{lemma2.2}, and the
Schauder theory of second-order parabolic PDEs, it follows that for
any $k$, $u_{1,n}(\cdot,t)$ and $u_{2,n}(\cdot,t)$ are locally (in
time) $C^{k}$ uniformly bounded with respect to $n$. Therefore,
there exists a subsequence (still denoted by $t_{n}$) such that
$u_{1,n}(\cdot,t)$ and $u_{2,n}(\cdot,t)$ converge locally uniformly
in any $C^{k}$ to $u_{1}^{*}(\cdot,t)$ and $u_{2}^{*}(\cdot,t)$
respectively, i.e.,
\begin{eqnarray*}
u_{1}^{*}(\cdot,t)=\lim \limits
_{n\rightarrow\infty}u_{1,n}(\cdot,t), \qquad
u_{2}^{*}(\cdot,t)=\lim \limits
_{n\rightarrow\infty}u_{2,n}(\cdot,t).
\end{eqnarray*}
Set $u^{*}:=u_{1}^{*}-u_{2}^{*}$, and then we have
\begin{eqnarray} \label{3.8}
\mathrm{osc}(u^{*})(t)&=&\mathrm{osc}(u_{1}^{*}-u_{2}^{*})\nonumber\\
&=&\lim \limits _{n\rightarrow\infty}\mathrm{osc}\left(u_{1}(x,t+t_{n})-\lambda t_{n}-u_{2}(x,t+t_{n})+\lambda t_{n}\right)\nonumber\\
&=&\lim \limits _{n\rightarrow\infty}\mathrm{osc}(u_{1}(x,t+t_{n})-u_{2}(x,t+t_{n}))\nonumber\\
&=&\lim \limits
_{n\rightarrow\infty}\mathrm{osc}(u)(t+t_{n})\nonumber\\
&=&\chi.
\end{eqnarray}
The second equality in (\ref{3.8}) holds since $u_{1,n}(\cdot,t)$
and $u_{1,n}(\cdot,t)$ are uniformly convergent.

Besides, it is easy to check that $u^{*}$ satisfies
\begin{eqnarray*}
 \left\{
\begin{array}{lll}
z_{t}=\sum\limits_{i,j=1}^{n}\widetilde{g}^{ij}z_{ij}+\sum\limits_{i=1}^{n}b_{i}z_{i} ~ \qquad \qquad &\mathrm{in} ~ \Omega\times(-\infty,\infty),\\
D_{\vec{\nu}}z=0 \qquad \qquad &\mathrm{on} ~
\partial \Omega\times(-\infty,\infty),
\end{array}
\right.
\end{eqnarray*}
where $\widetilde{g}^{ij}=g^{ij}(Du_{1}^{*})$
 and $b_{i}=\sum\limits_{k,l=1}^{n}(u_{2}^{*})_{kl}\int_{0}^{1}g^{kl}_{,p_{i}}(\eta
Du_{1}^{*}+(1-\eta)Du_{2}^{*})d\eta$. By the strong maximum
principle of second-order parabolic PDEs and Hopf's lemma, we know
$u^{*}$ is a constant. This is contradict with
$\mathrm{osc}(u^{*})(t)\equiv \chi$. Our \textbf{claim} follows. So,
one has $\lim \limits_{t\rightarrow\infty}\max_{\Omega}u=\lim
\limits_{t\rightarrow\infty}\min_{\Omega}u=c_{11}$ for some constant
$c_{11}$, which implies $\lim
\limits_{t\rightarrow\infty}|u-c_{11}|=0$. This finishes the proof
of Lemma \ref{main3.4}.
\end{proof}

Clearly, by Corollary \ref{main3.3} and Lemma \ref{main3.4}, we know
that the limit of any solution to the IBVP $(\sharp)$ is
$\widetilde{w}=w+\lambda t$ up to a constant, where
$(\lambda,\omega)$ is the solution to the BVP ($\ddag$). This
completes the proof of (2) of Theorem \ref{theorem1-1}.

\vspace{5mm}

\section*{Acknowledgments}
\renewcommand{\thesection}{\arabic{section}}
\renewcommand{\theequation}{\thesection.\arabic{equation}}
\setcounter{equation}{0} \setcounter{maintheorem}{0}

This research was supported in part by the NSF of China (Grant Nos.
11801496 and 11926352), the Fok Ying-Tung Education Foundation
(China), and Hubei Key Laboratory of Applied Mathematics (Hubei
University).


\begin{thebibliography}{99}


\bibitem{ar} U. Abresch, H. Rosenberg, \emph{A Hopf differential for constant mean curvature surfaces in $\mathbb{S}^{2}\times\mathbb{R}$ and
$\mathbb{H}^{2}\times\mathbb{R}$}, Acta Math. {\bf 193} (2004)
141--174.


\bibitem{aw} S.-J. Altschuler, L.-F. Wu, \emph{Convergence to translating solitons for a class of quasilinear
parabolic equations with fixed angle of contact to a boundary},
Math. Ann. {\bf 295} (1993) 761--765.


\bibitem{aw1} S.-J. Altschuler, L.-F. Wu, \emph{Translating surfaces of the non-parametric mean curvature flow with prescribed contact
angle}, Calc. Var. Partial Differential Equations {\bf 2} (1994)
101--111.



\bibitem{av1} S.-B. Angenent and J.-J.-L. Vel\'{a}zquez, \emph{Asymptotic shape of cusp singularities in curve shortening}, Duke
Math. J. {\bf 77} (1995) 71--110.

\bibitem{av2} S.-B. Angenent and J.-J.-L. Vel\'{a}zquez, \emph{Degenerate neckpinches in mean curvature
flow}, J. Reine Angew. Math. {\bf 482 } (1997) 15--66.


\bibitem{b1} C. Bandle, \emph{Isoperimetric inequalities and
applications}, Pitman, 1980.


\bibitem{chmx} L. Chen, D.-D. Hu, J. Mao, N. Xiang, \emph{Translating surfaces of the nonparametric mean curvature flow in Lorentz manifold
$M^{2}\times\mathbb{R}$}, available online at arXiv: 1804.10864v3.

\bibitem{gm2} Y. Gao, J. Mao, \emph{Translating solutions of the nonparametric mean curvature flow with
nonzero Neumann boundary data in product manifold
$M^{n}\times\mathbb{R}$, II}, preprint.



\bibitem{gms} Y. Gao, J. Mao, C.-L. Song, \emph{Existence and uniqueness of solutions to the constant mean curvature equation with nonzero Neumann boundary data in product manifold
$M^{n}\times\mathbb{R}$}, preprint.


\bibitem{gb} B. Guan, \emph{Mean curvature motion of nonparametric hypersurfaces with contact angle
condition}, Elliptic and parabolic methods in geoemtry (Minneapolis,
MN, 1994), pages 47--56, 1996.


\bibitem{gh1} G. Huisken, \emph{Flow by mean curvature of convex surfaces into
spheres}, J. Differential Geom. {\bf 20} (1984) 237--266.


\bibitem{gh2} G. Huisken, \emph{Contracting convex hypersurfaces
in Riemannian manifolds by their mean curvature}, Invent. Math. {\bf
84} (1986) 463--480.


\bibitem{gh3} G. Huisken, \emph{Non-parametric mean curvature evolution with
boundary conditions}, J. Differential Equat. {\bf 77} (1989)
369--378.


\bibitem{js} H. Jenkins, J. Serrin, \emph{The Dirichlet problem for the minimal surface equation in higher
dimensions}, J. reine angew. Math. {\bf 229} (1968) 170--187.



\bibitem{lmh} L. Mazet, M.-M. Rodr\'{\i}guez, H. Rosenberg, \emph{Periodic constant mean curvature surfaces in
$\mathbb{H}^{2}\times\mathbb{R}$}, Asian J. Math. {\bf 18} (2014)
829--858.




\bibitem{mww} X.-N. Ma, P.-H. Wang, W. Wei, \emph{Constant mean curvature surfaces and mean curvature flow with nonzero Neumann boundary conditions on strictly
convex domains}, J. Funct. Anal. {\bf 274} (2018) 252--277.


\bibitem{jm1} J. Mao, \emph{A new way to Dirichlet problem for minimal surface system in arbitrary dimensions and
codimensions}, Kyushu J. Math. {\bf 69}(1) (2015) 1--9.


\bibitem{mr} W.-H. Meeks III, H. Rosenberg, \emph{Stable minimal surfaces in
$M\times\mathbb{R}$},  J. Differential Geom. {\bf 68} (2004)
515--534.


\bibitem{hfj}H. Rosenberg, F. Schulze, J. Spruck, \emph{The half-space property
and entire positive minimal graphs in $M\times\mathbb{R}$}, J.
Differential Geom. {\bf 95} (2013) 321--336.


\bibitem{pt} P. Topping, \emph{Mean curvature flow and geometric
inequalities}, J. reine angew. Math. {\bf 503} (1998) 47--61.

\bibitem{o1} R. Osserman, \emph{The isoperimetric inequality}, Bull.
Amer. Math. Soc. {\bf 84} (1978) 1182--1238.


\bibitem{wwx} J. Wang, W. Wei, J.-J. Xu, \emph{Translating solutions of non-parametric mean curvature flows with capillary-type boundary value
problems}, Commun. Pure Appl. Anal. {\bf 18}(6) (2019) 3243--3265.


\bibitem{mtw1} M.-T. Wang, \emph{Long-time existence and convergence of graphic mean curvature flow in arbitrary
codimension}, Invent. Math. {\bf 148}(3) (2002) 525--543.

\bibitem{mtw2}  M.-T. Wang, \emph{The Dirichlet problem for the minimal surface system
in arbitrary codimension}, Commun. Pure. Appl. Math. {\bf 57}(2)
(2004) 267--281.

\bibitem{hyz} H.-Y. Zhou, \emph{Nonparametric mean curvature type flows of graphs with contact angle
conditions},  Int. Math. Res. Not. {\bf 19} (2018) 6026--6069.



\end{thebibliography}
 \end{document}